\documentclass{amsart}
\usepackage[centertags]{amsmath}
\usepackage{graphicx}
\usepackage{dsfont}
\usepackage{amssymb}
\usepackage{enumerate}
\usepackage[table]{xcolor}
\newtheorem{proposition}{Proposition}
\newtheorem{example}[proposition]{Example}
\newtheorem{lemma}[proposition]{Lemma}

\newtheorem{theorem}[proposition]{Theorem}

\newcommand{\II}{\mathds{I}}
\newcommand{\DD}{\mathcal{D}}
\begin{document}

\title{Extending multivariate sub-quasi-copulas}%

\author[D. Kokol B.]{Damjana Kokol Bukov\v{s}ek}
\address{Damjana Kokol Bukov\v{s}ek, University of Ljubljana, School of Economics and Business, and Institute of Mathematics, Physics and Mechanics, Ljubljana, Slovenia}
\email{damjana.kokol.bukovsek@ef.uni-lj.si}

\author[T. Ko\v sir]{Toma\v{z} Ko\v{s}ir}
\address{Toma\v{z} Ko\v{s}ir, University of Ljubljana, Faculty of Mathematics and Physics, and Institute of Mathematics, Physics and Mechanics, Ljubljana, Slovenia}
\email{tomaz.kosir@fmf.uni-lj.si}

\author[B. Moj\v{s}kerc]{Bla\v{z} Moj\v{s}kerc}
\address{Bla\v{z} Moj\v{s}kerc, University of Ljubljana, School of Economics and Business, and Institute of Mathematics, Physics and Mechanics, Ljubljana, Slovenia}
\email{blaz.mojskerc@ef.uni-lj.si}

\author[M. Omladi\v{c}]{Matja\v{z} Omladi\v{c}}
\address{Matja\v{z} Omladi\v{c}, Institute of Mathematics, Physics and Mechanics, Ljubljana, Slovenia}
\email{Matjaz@Omladic.net}


\begin{abstract}
  In this paper we introduce patchwork constructions for multivariate quasi-copulas. These results appear to be new since the kind of approach has been limited to either copulas or only bivariate quasi-copulas so far. It seems that the multivariate case is much more involved since we are able to prove that some of the known methods of bivariate constructions cannot be extended to higher dimensions. Our main result is to present the necessary and sufficient conditions both on the patch and the values of it for the desired multivariate quasi-copula to exist. We also give all possible solutions. 
\end{abstract}

\maketitle
\section{Introduction}\label{sec:intro}

Copulas are mathematical objects that capture the dependence structure among random variables. They may be viewed either as distributions with uniform margins or as building blocks of general distributions with given margins. Since they were introduced by Sklar in 1959 \cite{Skla} they have gained a lot of popularity through applications, say, in finance, insurance, and reliability theory. They are used both in probabilistic framework (cf.\ \cite{Nels,DuSe}) as well as in an imprecise setting (cf.\ \cite{MoMiMo,MoMiPeVi,OmSk,OmSt1,OmSt2,PeViMoMi1,PeViMoMi2}). An important concept related to copulas are quasi-copulas \cite{AlNeSc,DeBa,DiSaPlMeKl,GeQuMoRoLaSe,HaMe, JaDeBaDeMe,NeQuMoRoLaUbFl,NeQuMoRoLaUbFl2,Wall} which may be seen as the pointwise infima or suprema of copulas \cite[Section~7.3]{DuSe}. Precise definitions will be given in Section \ref{sec:patch}. 

The interest in quasi-copula theory and its applications is growing fast as can be seen from a recent review by Arias-Garcia et al. \cite{ArGaMeDeBa20} (cf.\ also Arias-Garcia et al. \cite{ArGaMeDeBa17} and Nelsen et al.\ \cite{NeQuMoRoLaUbFl}). With the development of copula theory in the imprecise probability setting the need for new construction methods for quasi-copulas has also increased. 
Montes et al., say, published in \cite{MoMiPeVi} a possible approach to bivariate imprecise copulas. If $\mathcal{C}$ is a nonempty set of copulas, then $\underline{C} = \inf_{C\in\mathcal{C}}\{C\}$ {and} $\overline{C}= \sup_{C\in \mathcal{C}}\{C\}$ are quasi-copulas and the ordered interval $(\underline{C}, \overline{C})$, i.e., the set of all intermediate quasi-copulas, may be called \emph{{an imprecise copula}}. Following the ideas of $p$-boxes it would be nice if the order ideal defined by these quasi-copulas contained a ``true'' copula. However, Omladi\v{c} and Stopar construct in \cite{OmSt1} an imprecise copula $(A,B)$ in this sense, i.e., an appropriate pair of quasi-copulas, such that there is no copula $C$ with $A \leqslant C\leqslant B$ (cf.\ also \cite{atslo,atslo1,OmSt3}). 
Many other needs for constructions of quasi-copulas include the search for the Dedekind-MacNeille completion of the poset (partially ordered set for pointwise order) of copulas. Nelsen and \'{U}beda-Flores showed in a historical paper \cite{NeUbFl} that in the bivariate case this completion is just the class of bivariate quasi-copulas. In the multivariate setting the two authors together with Fern\'{a}ndez-S\'{a}nchez demonstrate in \cite{FeSaNeUbFl} that for dimension $n\geqslant3$ a simple extension of this result to $n$-variate (quasi-)copulas does not work any more. The problem was recently solved by Omladi\v{c} and Stopar in \cite{OmSt5}, and their surprising solution is raising new questions in multivariate quasi-copulas.

So, it is high time to produce patchwork constructions for multivariate quasi-copulas. In copula setting this procedure is as old as the notion of copula itself. We want to obtain a copula from a certain function (called \emph{sub-copula}) defined on a suitable subset of $\II^n$, where $\II=[0,1]$, by extending it to a copula on the whole domain $\II^n$. We can think of sub-copula as a patch and search for copulas that extend their values given there. Indeed, even in the proof of Sklar's theorem one is faced with a version of this problem. With the development of the copula theory this technique has been used in quite a number of situations. For a list of applications and a full answer to it, including a necessary and sufficient conditions on the function and its domain as well as all possible solutions of the problem, an interested reader should consult the paper\ \cite{AmDiDuFe} by de Amo et al. (cf.\ also \cite{AmDiFe}). No need to emphasize that all this is done on the $n$-variate level.

On the other hand, the analogous problem for quasi-copulas is much newer and the known solutions appear to be limited only to the bivariate case, cf.\ the paper \cite{KoBuKoOmSt} by Kokol Bukov\v{s}ek et al.\ and the references given there. It seems that the multivariate case is much more involved since we are able to prove in this paper that some of the known methods of bivariate constructions cannot be extended to higher dimensions (Section \ref{sec:patch}). Nevertheless, our main result is to present the necessary and sufficient conditions both on the patch and the values of it for the desired multivariate quasi-copula to exist. We also give all the possible solutions (Section \ref{sec:ext}).

\section{Some partial results}\label{sec:patch}

We denote the unit interval by $\II:=[0,1]$ and the unit $n$-box by $\II^n$. Also, we denote by $R=\prod_{i=1}^{n} [a_i,b_i] \subseteq \II^n$ an arbitrary $n$-dimensional rectangle, i.e., a general $n$-box.
For a real function the term strictly increasing will be used in the sense increasing and increasing will be used instead of the more usual nondecreasing. Analogously for strictly decreasing and decreasing.

Consider subsets $\delta_1,\delta_2,\ldots,\delta_n$ of $\II$ that all contain $0$ and $1$ and let $\DD=\prod_{k=1}^{n}\delta_k$. A function $Q: \DD \to \II$ is called a \emph{sub-quasi-copula} (cf.\ \cite{QuMoSe}) if
\emph{\begin{enumerate}[(i)]
  \item $Q$ is grounded and 1 is the {neutral element} of $Q$;
  \item $Q$ is increasing in each of its arguments on $\DD$;
  \item $Q$ satisfies 1-Lipschitz condition on $\DD$ in each of its arguments
    .
\end{enumerate}}
Observe that in the case that $\delta_k=\II$ for all $k=1,2,\ldots,n$ so that $\DD=\II^n$ this is the usual definition of a quasi-copula.
Also, Condition \emph{(iii)} implies that $Q$ satisfies the Cauchy condition for any Cauchy sequence in $\DD$. Since we are only interested in continuous extensions of $Q$ we may and do assume that $\delta_k$ are closed subsets of $\II$ for all $k=1,2,\ldots,n$.

Assuming that the function is given 
on the outside of a given rectangle $R$ we want to extend it to $R\subseteq\II^n$ using an appropriate function to get a quasi-copula. The kind of procedure, sometimes called \emph{patchwork}, was first introduced in \cite{DeBaDeMe} for a class of functions containing copulas (cf.\ also \cite{DuFeSaQuMoUbFl} and \cite[p.~115]{DuSe}, and the references given there).
The main tool needed in the sequel, a multivariate extension of \cite[Theorem~1]{KoBuKoOmSt}, will be obtained via an inductive argument. We assume that a function is given on $\partial R$, i.e., on all $2n$ $(n-1)$-faces of $R$; in the beginning of this section we will call them simply facets. They are given in pairs: For each $k\in[n]$ we have the face defined by $x_k=a_k$ respectively  $x_k=b_k$, which we call the lower and upper $k$-th face and denote them by $R_k$ and $R_k'$.
We assume (inductively) that the values are given via $(n-1)$-variate functions: $F_k(\mathbf{x})$ respectively $F'_{k}(\mathbf{x}')$
\ for $\mathbf{x}\in R_k$ respectively $\mathbf{x}'\in R_k'$ for all $k\in[n]$.
By abuse of notation we will identify $R_k$ with $R_k'$ by viewing an $\mathbf{x}$ from the first one as corresponding to an $\mathbf{x}'$ from the latter one if they are equal in all coordinates except the $k$-th one.

Following the bivariate approach \cite[Section~2]{KoBuKoOmSt} we now introduce Conditions \textbf{PB} on a certain function defined on the $(k-1)$-faces of $R$. We will follow the notation above and denote the function by $F$ on lower faces and by $F'$ on the upper ones, with additional subscript identifying the actual face.

\textbf{Conditions PB} (\emph{Properties of the boundary functions}). The functions $F_k,F_k'$ for $k\in[n]$ have to satisfy the following conditions.
  \emph{\begin{enumerate}[(i)]
    \item The values of these functions are uniquely determined at $(n-2)$-faces of $R$ at which faces $R_k$ and $R_k'$ intersect.\footnote{This means that for every pair $j,k\in [n], j\neq k$, we have four equalities of the type, say, function $F_j(\mathbf{x})$ at $\mathbf{x}\in R_j$ with $x_k=a_k$ equals $F_k(\mathbf{x})$ at $\mathbf{x}\in R_k$ with $x_j=a_j$.}
    \item These functions are increasing and 1-Lipschitz in each variable.
    \item For all $k\in[n]$ and all $\mathbf{x}\in R_k$ we have $0\leqslant F_k'(\mathbf{x})-F_k(\mathbf{x}) \leqslant b_k-a_k$. {
    \item For any $k\in[n]$ the volume of the rectangle $R_{\mathbf{x}} \subseteq R$ defined by the points $\mathbf{a}$ and $\mathbf{x}\in R_k'$ is a monotone and non-constant function in each argument of $\mathbf{x}$.\footnote{In case $n=2$ this condition reduces exactly to the additional assumption immediately following \textbf{Conditions PB} of \cite{KoBuKoOmSt}. This assumption will be needed in the proof of 
        Lemma \ref{le:additive}\emph{(d)} and in Example \ref{ex:old},
        and is not crucial in the rest of the paper.}}
  \end{enumerate}}

The multivariate case is substantially more involved than the bivariate one, so that we need to work in two steps.

\vskip5mm
\begin{center}
  {\textsc Step I}
\end{center}
\vskip5mm

On the first step we assume that $R=\II^n$ and we are searching for a quasi-copula. So, in Conditions \textbf{PB} we assume functions $F_k$ to be zero and functions $F'_k =C_k(x_1,x_2,\ldots,x_{\hat{k}},\ldots,x_n)$, where notation $\hat{k}$ means that index $k$ is not present in this particular counting. Here is the desired quasi-copula.

\begin{proposition}\label{pr:stepI}
  Let Conditions \textbf{PB} be satisfied. Then, there exists a quasi-copula $Q$ matching functions $F_k,F'_k$ on the boundary of $\II^n$. {Moreover,
  \begin{enumerate}[(a)]
    \item The upper bound of all these solutions equals
    \[
       \overline{Q}(x_1,x_2,\ldots,x_n)= \]
  \[
       \min\{C_1(x_2,\ldots,x_n),C_2(x_1,x_3,\ldots,x_n), \ldots, C_n(x_1,\ldots,x_{n-1}) \},
  \]
    \item the lower bound equals
  \[
       \underline{Q}(x_1,x_2,\ldots,x_n)= \]
  \[
       \max\{W(x_1,C_1(x_2,\ldots,x_n)),W(x_2,C_2(x_1,x_3,\ldots,x_n)), \ldots, W(x_n,C_n(x_1,\ldots,x_{n-1})) \},
  \]
    \item and every function $Q$ satisfying Condition \textbf{PB}\emph{(ii)} and $\underline{Q} \leqslant Q\leqslant \overline{Q}$ solves the problem.
  \end{enumerate}}
\end{proposition}

{\begin{proof}
First we show \emph{(b)}. Clearly, each of the terms in this maximum is a quasi-copula, so that also $\underline{Q}$ is a quasi-copula. Furthermore, we need to see that the boundary conditions on the faces are satisfied as well. The faces containing a zero are easy again. So, assume first $x_1=1$. In this case the first term equals $C_1(x_2,\ldots,x_n)$. The behavior of any other term can be seen from the behavior of the term at $k=2$ which is equal to $W(x_2,C_2(1,x_3,\ldots,x_n)) = W(x_2,C_1(1,x_3,\ldots,x_n)) \leqslant C_1(x_2,\ldots,x_n)$, because $C_1(1,x_3,\ldots,x_n)- C_1(x_2,\ldots,x_n) \leqslant 1-x_2$ by the 1-Lipschitz condition. Assertion \emph{(a)} goes similarly, perhaps even somewhat easier. These two facts imply the third one in a straightforward way.
\end{proof}
}


\vskip5mm
\begin{center}
  {\textsc Step II}
\end{center}
\vskip5mm

We now start with the general problem. Let the functions $F_k,F'_k$ satisfy Conditions \textbf{PB} for $k\in[n]$. We want to find a patch of a quasi-copula that will satisfy these conditions.
By abuse of notation we first introduce formally a function $F$ of $n$ variables from $\partial R$ such that its value at $x_k=a_k$ equals $F_k$ and such that its value at $x_k=b_k$ equals $F_{k'}$ for $k\in[n]$. Since we will always need this function only at all points that belong to the boundary $\partial R$, there will be no ambiguity in this notation even if such function defined on all $R$ does not exist. For an arbitrary $\mathbf{x}\in R$ and $z\in\{-1,0,1\}^{[n]}$ let $\mathbf{x}^z$ be the $n$-tuple whose $k$-th coordinate is equal to $( \mathbf{x}^z)_k= \left\{
                                            \begin{array}{ll}
                                              b_k, & \hbox{if $z(k)=1$;} \\
                                              x_k, & \hbox{if $z(k)=0$;} \\
                                              a_k, & \hbox{if $z(k)=-1$;}
                                            \end{array}
                                          \right.$ for $k\in[n]$. Here, $z(k)$ denotes the $k$-th coordinate of the $n$-tuple $z$; notations s.a.\ $z_k$ will be reserved for particular $n$-tuples. {Because of the requirement above that $F$ be defined on $\partial R$ only, the value $F(\mathbf{x}^z)$ will always be defined as soon as we insist that $z\neq0$.} We first introduce the additive part of our patch
\[
    A(x_1,x_2,\cdots,x_n)=\sum_{0\neq z\in\{-1,0\}^{[n]}} (-1)^{1+\sum z} F(\mathbf{x}^z).{\quad\quad\mbox{(A)}}
\]
which matches the initial conditions on the faces containing the ``$0$'' point $\mathbf{a}= (a_1,\ldots,a_n)$ as shown in the following lemma. We also need the function matching the initial conditions on the faces containing the ``$1$'' point $\mathbf{b}= (b_1,\ldots,b_n)$
\[
    B(x_1,x_2,\cdots,x_n)=\sum_{0\neq z\in\{0,1\}^{[n]}} (-1)^{1+\sum z} F(\mathbf{x}^z) 
{\quad\quad\mbox{(B)}}
\]
and their difference $G(\mathbf{x})=B(\mathbf{x})-A(\mathbf{x})$.

\begin{lemma}\label{le:additive} Denote by $V$ the volume of rectangle $R$, suppose that $V\ne0$, and let $k\in[n]$ be arbitrary.
  \begin{enumerate}[(a)]
    \item Let $z_k'$ be an $n$-tuple of indices with $1$ in the $k$-th position and zeros elsewhere. Then $B(\mathbf{x}^{z_k'})= F'_{k}(x_1, \ldots, x_{\hat{k}},\ldots,x_n)$, $B(\mathbf{b})=F(\mathbf{b})$, and $B(\mathbf{a})=F(\mathbf{a})-(-1)^nV$.
    \item Let $z_k(=-z_k')$ be an $n$-tuple of indices with $-1$ in the $k$-th position and zeros elsewhere. Then $A(\mathbf{x}^{z_k})= F_k(x_1, \ldots, x_{\hat{k}},\ldots, x_n)$, $A(\mathbf{a})=F(\mathbf{a})$, and $A(\mathbf{b})=F(\mathbf{b})-V$.
    \item Let $\hat{z}_k$ be an $n$-tuple of indices with zero in the $k$-th position and ones elsewhere. Then for $M_k(x_k)= G(\mathbf{x}^{\hat{z}_k})$ we have $M_k(a_k)=0$ and $M_k(b_k)=V$.{
    \item Functions $\dfrac{M_k(x_k)}{V}$ are increasing and are sending interval $[a_k,b_k]$ to $[0,1]$.}
  \end{enumerate}
\end{lemma}

\begin{proof}
  For every $k\in[n]$ denote by $\mathcal{N}_k$ the set of all subsets of $[n]$ that do not contain $k$ and by $\mathcal{N}'_k$ the set of all subsets of $[n]$ that do contain $k$. There is a natural bijection 
$\mathcal{N}_k\longrightarrow \mathcal{N}'_k,
M\mapsto M'= M\cup\{k\}$. The two sets are disjoint and their union equals $2^{[n]}$.
\emph{(a)} Using definition (B) we have
\[
    \begin{split}
       B(\mathbf{x}^{z_k'}) & = \sum_{\substack{z\in\{0,1\}^{[n]}\\0\neq z,z(k)=0}} (-1)^{1+\sum z} F\left( (\mathbf{x}^{z_k'})^{z} \right) + \sum_{\substack{z\in\{0,1\}^{[n]}\\z(k)\neq0}} (-1)^{1+\sum z} F\left( \mathbf{x}^{z}  \right).
    \end{split}
\]
Observe that the $n$-tuples $z$ of the left hand sum have supports in $\mathcal{N}_k$ and the ones of the right hand sum have supports in $\mathcal{N}'_k$. Consider a term of the left hand sum and the corresponding $z$, and denote by $M$ its support, i.e., the nonempty set of all indices at which it is nonzero. Clearly, $M\in\mathcal{N}_k$. Then, $(\mathbf{x}^{z_k'})^{z}=\mathbf{x}^{\tilde{z}}$, where $\tilde{z}$ is exactly the $n$-tuple whose value is $1$ on $M'\in\mathcal{N}'_k$ and zero elsewhere. When $z$ runs through all $n$-tuples of the left hand side sum, $\tilde{z}$ runs through all the $n$-tuples of the right hand side sum except for the term corresponding to the empty support. Observe that the term corresponding to $\tilde{z}$ has exactly the opposite sign than the term corresponding to $z$, so that they sum up to zero. Since $\emptyset'=\{k\}$ we conclude
\[
    \begin{split}
        B(\mathbf{x}^{z_k'}) & = F(\mathbf{x}^{z_{k}'})=F'_{k}(x_1, \ldots, x_{\hat{k}},\ldots,x_n)
    \end{split}
\]
as desired. A simple computation reveals that
\[
    V=\sum_{z\in\{0,1\}^{[n]}}(-1)^{n-\sum z}F(\mathbf{a}^z),
\quad\mbox{and}\quad
    B(\mathbf{b})=\sum_{0\neq z\in\{0,1\}^{[n]}} (-1)^{1+\sum z} F(\mathbf{b}^z)= F(\mathbf{b}),
\]
since $F(\mathbf{b}^z)= F(\mathbf{b})$ and the sum of the corresponding factors equals $1$. Finally,
\[
    B(\mathbf{a})=F(\mathbf{a})+\sum_{0\neq z\in\{0,1\}^{[n]}} (-1)^{1+\sum z} F(\mathbf{a}^z)=\sum_{z\in\{0,1\}^{[n]}} (-1)^{1+\sum z} F(\mathbf{a}^z)
\]
which equals $(-1)^{n-1}V$ by the above. \emph{(b)} This time we use definition (A) to get
\[
    A(\mathbf{x}^{z_k})=\sum_{\substack{z\in\{-1,0\}^{[n]}\\0\neq z,z(k)=0}} (-1)^{1+\sum z} F\left((\mathbf{x}^{z_k})^z\right) + \sum_{\substack{z\in\{-1,0\}^{[n]}\\z(k)\neq0}} (-1)^{1+\sum z} F\left(\mathbf{x}^z\right)
\]
Following the considerations as in \emph{(a)} we observe that a given $n$-tuple $z$ of the
left hand sum has its support $\emptyset\neq M\in\mathcal{N}_k$. So, $(\mathbf{x}^{z_k})^{z} =\mathbf{x}^{\tilde{z}}$, where $\tilde{z}$ is exactly the $n$-tuple whose value is $-1$ on $M'\in\mathcal{N}'_k$ and zero elsewhere. When $z$ runs through $n$-tuples of the kind, $\tilde{z}$ runs through all the $n$-tuples of the right hand side sum except for the term corresponding to the empty support. Observe that the term of $\tilde{z}$ is exactly the opposite to the term of $z$, so that they sum up to zero. Consequently
\[
    \begin{split}
        A(\mathbf{x}^{z_k}) & = F(\mathbf{x}^{z_{k}})=F_{k}(x_1, \ldots, x_{\hat{k}},\ldots,x_n)
    \end{split}
\]
as desired. Through a simple computation we learn that
\[
    V=\sum_{z\in\{-1,0\}^{[n]}}(-1)^{n+\sum z}F(\mathbf{b}^z),
\quad\mbox{and}\quad
    A(\mathbf{a})=\sum_{0\neq z\in\{-1,0\}^{[n]}} (-1)^{1+\sum z} F(\mathbf{a}^z)= F(\mathbf{a}),
\]
since $F(\mathbf{a}^z)= F(\mathbf{a})$ and the sum of the corresponding factors equals $1$. Finally,
\[
    A(\mathbf{b})-F(\mathbf{b}) = \sum_{z\in\{-1,0\}^{[n]} } (-1)^{1+\sum z} F(\mathbf{b}^z)= 
    -V.
\]
\emph{(c)}
At the upper end we have  $M_k(b_k)= G(\mathbf{b}^{\hat{z}_k})=G(\mathbf{b})=B(\mathbf{b})- A(\mathbf{b})=F(\mathbf{b})-F(\mathbf{b})+V=V$,
where we used \emph{(a)} and \emph{(b)}. At the lower end we obtain first that $M_k(a_k)= G(\mathbf{a}^{\hat{z}_k})=B(\mathbf{a}^{\hat{z}_k})- A(\mathbf{a}^{\hat{z}_k})$, so that
\[
    M_k(a_k)= \sum_{0\neq z\in\{0,1\}^{[n]}} (-1)^{1+\sum z} F\left(\left(\mathbf{a}^{\hat{z}_k}\right)^z\right)- \sum_{0\neq z\in\{-1,0\}^{[n]}} (-1)^{1+\sum z} F\left(\left(\mathbf{a}^{\hat{z}_k}\right)^z\right).
\]
Now, recall the methods of the proof of \emph{(a)}. The first one of these two sums can be written in two parts depending on whether $z(k)=0$ or $z(k)=1$. Each term in the second part has a corresponding term in the first one of the same value and the opposite sign, except for $z=z_k$, so that the first sum equals $F\left(\mathbf{a}^{\hat{z}_k}\right)$. In an analogous way we conclude that the second sum equals $F\left(\mathbf{a}^{\hat{z}_k}\right)$, thus proving the desired.

{
\emph{(d)}
 For $x_k\in(a_k,b_k)$ we have $M_k(x_k)= G(\mathbf{x}^{\hat{z}_k})=B(\mathbf{x}^{\hat{z}_k})- A(\mathbf{x}^{\hat{z}_k})$, so that
\[
    \begin{split}
       M_k(x_k) & = \sum_{0\neq z\in\{0,1\}^{[n]}} (-1)^{1+\sum z} F\left(\left(\mathbf{x}^{\hat{z}_k}\right)^z\right)- \sum_{0\neq z\in\{-1,0\}^{[n]}} (-1)^{1+\sum z} F\left(\left(\mathbf{x}^{\hat{z}_k}\right)^z\right) \\
         & = \sum_{\substack{z\in\{0,1\}^{[n]}\\z_k=1}} (-1)^{1+\sum z} F\left(\mathbf{b}\right)
         + \sum_{\substack{0\neq z\in\{0,1\}^{[n]}\\z_k=0}}(-1)^{1+\sum z} F(\mathbf{x}^{\hat{z}_k}) \\
         & - \sum_{\substack{z\in\{-1,0\}^{[n]}\\z_k=-1}}(-1)^{1+\sum z} F\left(\left(\mathbf{x}^{\hat{z}_k}\right)^z\right)
         - \sum_{\substack{0\neq z\in\{-1,0\}^{[n]}\\z_k=0}} (-1)^{1+\sum z} F\left(\left(\mathbf{x}^{\hat{z}_k}\right)^z\right).
    \end{split}
\]
The first two sums above add up to $F(\mathbf{x}^{\hat{z}_k})$ after observing that in the first sum a constant is multiplied by the same number of positive and negative signs which add up to zero, and in the second sum the number of positive signs is greater by one so that they add up to $1$. }{
For each $z$ in the third sum we have $\left(\mathbf{x}^{\hat{z}_k}\right)^z=\mathbf{b}^z$, so that it adds up to constant
\[
    - \sum_{\substack{z\in\{-1,0\}^{[n]}\\z_k=-1}}(-1)^{1+\sum z} F(\mathbf{b}^z).
\]
Combine these observations into
\[
    \begin{split}
       M_k(x_k) & = \sum_{\substack{z\in\{-1,0\}^{[n]}\\z_k=-1}}(-1)^{\sum z} F(\mathbf{b}^z)  + F(\mathbf{x}^{\hat{z}_k}) \\
         & + \sum_{\substack{0\neq z\in\{-1,0\}^{[n]}\\z_k=0}} (-1)^{\sum z} F(\left(\mathbf{x}^{\hat{z}_k}\right)^z).
    \end{split}
\]
Now, on the right hand side of the last expression above we get exactly the volume of the rectangle determined by points $\mathbf{a}$ and $\mathbf{x}^{\hat{z}_k}$. By Condition \textbf{PB}\emph{(iv)} this volume is a non-constant monotone function of $x_k$ and the desired conclusion follows by \emph{(c)}.}
\end{proof}
{
Observe that functions $G_k(x_1, \ldots, x_{\hat{k}},\ldots,x_n)= G(\mathbf{x}^{z_k'})$ for $k\in[n]$ are exactly the volumes of rectangles $R_\mathbf{x}\subseteq R$ defined by the points $\mathbf{a}, \mathbf{x}\in R$. These volumes form, by Condition \textbf{PB}\emph{(iv)}, a non-constant function of $\mathbf{x}$ monotone in each of its coordinates thus constituting a quasi-distribution in the sense of \cite[Section~3]{OmSt4} (after being divided by $V$) with marginal distributions $\dfrac{M_j(x_j)}{V}$ for $j\in[n],j\ne k$. Denote by $Q_k$ the quasi-copulas for $k\in[n]$ obtained by the Sklar type theorem of \cite[Theorem 8]{OmSt4}. Finally, let $Q$ be the quasi-copula obtained from quasi-copulas $Q_k$ for $k\in[n]$ by Proposition \ref{pr:stepI}.}

{
We are studying a patch of quasi-copula type that is equal to $A$ on $(n-1)$-faces containing point $\mathbf{a}$ and is equal to $B$ on $(n-1)$-faces containing point $\mathbf{b}$. Let the univariate margins $M_k$ for $k\in[n]$ be defined as above. Then, under the condition that the volume $V$ of the rectangle $R$ is nonzero (actually this follows from \textbf{Condition PB}\emph{(iv)}), one might conjecture that the desired patch is equal to
\begin{equation*}\label{eq:proti}
  P(x_1,x_2,\cdots,x_n) = A(x_1,x_2,\ldots,x_n)+ VQ\left(\dfrac{M_1(x_1)}{V} ,\cdots,\dfrac{M_n(x_n)}{V} \right),
\end{equation*}
for any quasi-copula $Q$. For $n=2$ this is true by \cite[Theorem 1]{KoBuKoOmSt}. However, for $n\geqslant3$ this does not hold in general as the following example exhibits. Nevertheless we will be able to solve the proposed problem in higher dimensions as well using independent methods and only Conditions \textbf{PB}\emph{(i)}, \emph{(ii)} and \emph{(iii)}.

\begin{example}\label{ex:old}
\end{example}
We first divide the unit square into $9$ equal squares. Denote $\mathcal{M}=\{(x,y); \dfrac{1}{3}\leqslant x,y\leqslant\dfrac{2}{3}\}$ and let $\mathcal{C}$ be the union of the corners $\{(x,y); 0\leqslant x\leqslant \dfrac{1}{3}, \dfrac{2}{3} \leqslant y\leqslant1\}$, $\{(x,y); 0\leqslant x,y\leqslant\dfrac{1}{3}\}$, $\{(x,y); \dfrac{2}{3} \leqslant x,y\leqslant1\}$, and $\{(x,y); \dfrac{2}{3} \leqslant x\leqslant1, 0\leqslant y\leqslant \dfrac{1}{3}\}$. Furthermore, let $\mathcal{D}$ be the union of the four squares not yet taken into account so far. Define a bivariate quasi-copula $D(x,y)$ with the density function
\[
   d(x,y)=\begin{cases}
            -3, & \mbox{on } \mathcal{M}, \\
            0, & \mbox{on } \mathcal{C}, \\
            3, & \mbox{on } \mathcal{D},
          \end{cases}
\]
and a $3$-dimensional quasi-copula $F(x,y,z)=D(x,y)z$. Next, we consider 3-dimensional patch $Q$ on the rectangle $R=\left[\dfrac{1}{3},\dfrac{2}{3}\right]^3$. We have on the boundary of $R$
\[
   \begin{split}
      F\left(\dfrac{1}{3},y,z\right)=\left(y-\dfrac{1}{3}\right)z,\quad  & F\left(\dfrac{2}{3},y,z\right)=\dfrac{1}{3}z, \\
      F\left(x,\dfrac{1}{3},z\right)=\left(x-\dfrac{1}{3}\right)z,\quad  & F\left(x,\dfrac{2}{3},z\right)=\dfrac{1}{3}z, \\
      F\left(x,y,\dfrac{1}{3}\right)=\dfrac{2}{3}(x+y)-\left(\dfrac{1}{3}+xy\right),\quad  & F\left(x,y,\dfrac{2}{3}\right)=\dfrac{4}{3}(x+y)-\left(\dfrac{2}{3}+2xy\right).
   \end{split}
\]
After we insert these values into functions $A$, $B$, and $G$ given above we obtain
\[
   \begin{split}
      A(x,y,z)  & = yz +xz-xy +\dfrac{1}{3}(x+y-2z)-\dfrac{1}{9}, \\
      B(x,y,z)  & = -2xy +\dfrac{1}{3}(4x+4y+z)-\dfrac{8}{9},\ \mbox{and} \\
      G(x,y,z)  & = x+y+z-xy-xz-yz-\dfrac{7}{9}.
   \end{split}
\]
Then,
\[
   M_1(x)=-\dfrac{x}{3}+\dfrac{1}{9},\ \ M_2(y)=-\dfrac{y}{3}+\dfrac{1}{9},\ \ \mbox{and} \ \ M_3(z)=-\dfrac{z}{3}+\dfrac{1}{9},\ \
\]
so that the patch on $R$ is given by
\[
   P(x,y,z)=A(x,y,z)-\dfrac{1}{9}Q(3x-1,3y-1,3z-1),
\]
where $Q$ is a quasi-copula with appropriate bivariate marginals. Notice that
\[
   \dfrac{G_1(y,z)}{V}=\dfrac{G(\frac23,y,z)}{V}=\dfrac{M_2(y)}{V}\dfrac{M_3(z)}{V},
\]
and similarly
\[
   \dfrac{G_2(x,z)}{V}=\dfrac{M_1(x)}{V}\dfrac{M_3(z)}{V}, \dfrac{G_3(x,y)}{V}=\dfrac{M_1(x)}{V}\dfrac{M_2(y)}{V}.
\]
Hence bivariate quasi-copulas obtained by the (extension of) Sklar theorem are $Q_1=Q_2=Q_3=\pi$, the independence copula. Apply Proposition \ref{pr:stepI} with $C_i=Q_i$ for $i=1,2,3$ to get
\[
   \underline{Q}(x,y,z)=\max\{W(x,yz),W(y,xz),W(z,xy)\}=\max\{0,x+yz-1,y+xz-1,z+xy-1\}.
\]
In this case $P$ is not a quasi-copula. For instance,
\[    P\left(x,\dfrac{1}{2},\dfrac{2}{5}\right)=\min \left\{\dfrac{7x}{30}-\dfrac{1}{90},-\dfrac{x}{10}+ \dfrac{1}{5} \right\} \]
is a function of $x$ which is not increasing for $x\in\left(\dfrac{19}{30},\dfrac{2}{3}\right)$.
Hence $P$ is not a patch of a quasi-copula.

}




\section{Extending a sub-quasi-copula}\label{sec:ext}

In this section we extend a sub-quasi-copula to a quasi-copula.
On the way to independent solutions of this problem we first recall that an \emph{order closed} set is meant to be closed under the operation of suprema and infima of arbitrary subsets.

\begin{lemma}\label{le:infsup}
  Let $\mathcal{S}$ be a bounded set of functions on a rectangle $R$ such that each of them is increasing and 1-Lipschitz in each argument. Then the local bounds
\[
    \underline{\mathcal{S}}=\inf\{S\in\mathcal{S}\}\quad\mbox{and}\quad \overline{\mathcal{S}} = \sup\{S\in\mathcal{S}\}
\]
are both increasing and 1-Lipschitz in each argument. The order closure of $\mathcal{S}$ is compact in the uniform norm.
\end{lemma}



While the proof of this lemma is a straightforward extension of the proof of \cite[Theorem~6.2.5]{Nels}
, the extension to be given now 
is more sophisticated. Recall the notation for the formal boundary function $F$, and index functions $z_k$ and $z_k'$ of Lemma \ref{le:additive}.

\begin{proposition}[Local patch bounds]\label{pr:mejekrp} Fix a choice of function $F$ satisfying the boundary conditions \textrm{\bf PB} (i)--(iii). Then the set $\mathcal{S}$ of patches having $F$ for its boundary, and being increasing and 1-Lipschitz in each argument, is nonempty. The local bounds of  $\mathcal{S}$ on $R$ are equal to
\[
  \begin{split}
     \overline{\mathcal{S}} & = \min_{k\in[n]}\{ F(\mathbf{x}^{z_k'}), F(\mathbf{x}^{z_k})+x_k-a_k \} \\
     \underline{\mathcal{S}} & = \max_{k\in[n]}\{ F(\mathbf{x}^{z_k}),  F(\mathbf{x}^{z_k'})+x_k-b_k \}.
  \end{split}
\]
\end{proposition}

\begin{proof}
  Denote by $P$ a hypothetical desired patch and note that for all $k\in[n]$: $P(\mathbf{x}) \leqslant F(\mathbf{x}^{z_k'})$ and $P(\mathbf{x})- F(\mathbf{x}^{z_k}) \leqslant x_k-a_k$, first by monotonicity and second by the 1-Lipschitz property of the patch. This implies easily that $P\leqslant \overline{\mathcal{S}}$ in the pointwise order. The fact that $\overline{\mathcal{S}}$ is clearly monotone increasing and 1-Lipschitz in each argument will then imply that it is a possible patch, as soon as we show that it satisfies the boundary conditions. Actually, we will show that
\begin{equation}\label{eq:meja}
  \overline{\mathcal{S}}(\mathbf{x}^{z_j})=F(\mathbf{x}^{z_j})
\end{equation}
for any $j\in[n]$. To prove that, fix $j$ and let us go through all the terms $k\in[n]$ of the defining minimum of $\overline{\mathcal{S}}$. When $k=j$, the two terms bring us to the desired Equation \eqref{eq:meja} since $x_j-a_j=0$ and $F(\mathbf{x}^{z_j}) \leqslant F(\mathbf{x}^{z_j'})$ by Condition \textbf{PB}\emph{(iii)}. Now, if we consider an index $k\ne j$, we are in an $(n-1)$-dimensional face of $R$, where function $F$ is defined, monotone increasing and 1-Lipschitz by condition \textbf{PB}\emph{(ii)}. The left hand term of the minimum becomes in this case $F\left((\mathbf{x}^{z_j})^{z_k'}\right)\geqslant F\left(\mathbf{x}^{z_j}\right)$, so that \eqref{eq:meja} holds in this case as well. Finally, to get the right hand side term we compute
\[
    F\left(\mathbf{x}^{z_j}\right) - F\left(\left(\mathbf{x}^{z_j}\right)^{z_k}\right) \leqslant x_k-a_k,
\]
thus concluding the proof of the first half of the ``upper'' part of the proposition. The second half goes in exactly the same way: We need to show that Equation \eqref{eq:meja} holds when we replace $z_j$ by $z_j'$ on both sides. The fact that $\underline{\mathcal{S}}$ is the lower bound of this set goes in a similar way.
\end{proof}


Suppose that $Q^*:\DD\to\II$ is a sub-quasi-copula. By \cite[Theorem 2.3]{QuMoSe} it has an extension to a quasi-copula $Q \colon \II^2\to\II$ in case of $n=2$; in \cite{KoBuKoOmSt} all the solutions were found in this case. However, for the multivariate version of an equivalent problem only the copula case was solved (cf.\ \cite{AmDiFe,AmDiDuFe}).
Let us now extend the notation of \cite{KoBuKoOmSt}. Let the considered sub-quasi-copula be defined on $\DD=\delta_1\times\delta_2\times\cdots\times \delta_n$, where the closed sets $\delta_i \subseteq\II$ are given for $i\in[n]$. Also for $i\in[n]$ let us introduce:
\emph{\begin{enumerate}[(i)]
  \item Write the sets $\delta_i$ as a union of singletons and closed intervals. Consider only maximal possible intervals contained in $\delta_i$ and observe that any two of the kind are disjoint. Denote by $\mathcal{I}_i=\{I_j^i\}_{j\in J_i}$ the (necessarily countable) set of all such nontrivial intervals.
  \item Write the complement $\II\setminus \delta_i$ as a union of countably many disjoint open intervals. We introduce the family $\mathcal{O}_i=\{O_j^i\}_{j\in K_i}$ of all the closures of these intervals.
  \item Introduce $\mathcal{T}_i=\mathcal{I}_i\cup\mathcal{O}_i$ and $L_i=J_i\sqcup K_i$.
\end{enumerate}}
Observe that all the index sets $J_i,K_i$, and $L_i$ are countable. Also, any pair of the closed intervals in the family $\mathcal{T}_i$ is clearly either disjoint or intersects at a common endpoint. 
By abuse of notation we will identify a family of intervals by the union of these intervals. Observe that $\mathcal{T}_i$ is dense in $\II$. Indeed, assume, if possible, that $\II$ contains an open interval $I$ having an empty intersection with $\mathcal{T}_i$. Therefore, $I$ has an empty intersection with $\mathcal{O}_i$, so that it is contained in $\delta_i$. Consequently, there is a closed interval containing $I$ and being a member of $\mathcal{I}_i$, contradicting the above.
 Note that the complement of $\mathcal{T}_i$ in $\II$ can be a countable or even uncountable Cantor like set with positive Lebesgue measure in [0,1] (cf.\ \cite{GeOl,AmDiDuFe}).

Let us denote by $\mathcal{T}$ the set of all rectangles $R_\mathbf{s}= \displaystyle \prod_{\substack{i=1
}}^{n}[a_{s_i},b_{s_i}]$ indexed by $n$-tuples of indices $\mathbf{s}=(s_1,s_2,\ldots,s_n)$ with $s_i\in L_i$ for $i\in[n]$ and denote the set of these $n$-tuples by $L$. Observe that necessarily $[a_{s_i},b_{s_i}]$ is either a member of $\mathcal{I}_i$ or a member of $\mathcal{O}_i$ and that in either case $a_{s_i}\neq b_{s_i}$, so that a rectangle $R_{\mathbf{s}}$ for $\mathbf{s}\in L$ is always nondegenerate. It follows easily from our definitions that the family of rectangles  $R_{\mathbf{s}}$ for $\mathbf{}s\in L$ is countable and dense in $\II^n$.
Following our notation of Section \ref{sec:patch} we introduce a hypothetical function $F_\mathbf{s}$ to be defined on the boundary of the rectangle $R_\mathbf{s}$ for every $\mathbf{s}\in L$. It is almost clear from the above that $F$ is defined at least at the vertices of $R_\mathbf{s}$, i.e., at points $\mathbf{x}^z$ where {vector} 
$z$ has no zero entries. Let us look into some details about where exactly this function is defined
. In the defining product $R_\mathbf{s}= \displaystyle \prod_{\substack{i=1
}}^{n}[a_{s_i},b_{s_i}]$ some of the segments (say, $k$ of them, $0\leqslant k\leqslant n$) belong to $\mathcal{I}_\cdot$ and some to $\mathcal{O}_\cdot$. Let us rewrite it so that the former ones are counted with starting indices continuing with counting the latter ones, after an appropriate permutation of indices:
\[
    \widetilde{R}_\mathbf{s}^k= \displaystyle \prod_{\substack{i=1
}}^{n}[a_{s_i},b_{s_i}],\ \mbox{where}\ [a_{s_i},b_{s_i}]\in\mathcal{I}_i, 1\leqslant i\leqslant k,\ \mbox{and}\ [a_{s_i},b_{s_i}]\in\mathcal{O}_i, k< i\leqslant n.
\]
Consequently, the definition set of the hypothetical function $F$ contains exactly the members of $\widetilde{R}_\mathbf{s}^k$ of the form $\mathbf{x}^z$ with $z_i$ nonzero whenever $k< i\leqslant n$. Clearly, the definition set of $F$ on the original rectangle $R_\mathbf{s}$ can be obtained from this one using the reverse permutation of indices on this result. We will denote it by $D_F$.
It is our aim to find a possible quasi-copula $Q$ solving our problem using three main ingredients: (1) use induction on dimension $k$ for which $Q^*$ is defined on a given $R_\mathbf{s}$ for $\mathbf{s}\in L$; (2) extend these solutions on all of $\mathcal{T}$ using the fact that $L$ is countable, and (3) extend this solution by continuity to all of $\II^n$. 

We observed in Section \ref{sec:patch} that it is good to have existence of a patch given the boundary condition but it is even better to have an upper and a lower bound for all the patches of the kind. It is our aim now to present an algorithm for the two bounds of all quasi-copulas extending a given sub-quasi-copula.
We start by finding these bounds on segments and organize the work by going through segments along each of the coordinates. Actually, we will give particularities only for the $n$-th coordinate since any other goes in exactly the same way. Consider all vertices of the $(n-1)$-dimensional rectangles of the form $\displaystyle \prod_{\substack{i=1
}}^{n-1}[a_{s_i},b_{s_i}]$ -- they build a mesh to be denoted by $V$.
Observe that all these vertices belong to (the corresponding coordinates of) the definition set of $Q^*$ since they are either in the intervals of type $\mathcal{I}$ or at the endpoints of the intervals of type $\mathcal{O}$.
Rewrite the countable set $\cup_{t\in L_n}\{a_{t},b_t\}$ as $\{c_t\}_{t\in M}$ in order to simplify the notation to follow and denote $\gamma_t^\mathbf{v}=Q^*(\mathbf{v},c_t)$ for all $\mathbf{v}\in V$ and $t\in M$. In the following proposition we will need these values to satisfy
\begin{enumerate}[(A)]
  \item $0\leqslant\gamma_t^{\mathbf{v}'}-\gamma_t^{\mathbf{v}}\leqslant v'_i-v_i$ for all $t\in M$ and $\mathbf{v}',\mathbf{v}\in V$ such that they are distinct only in the $i$-th coordinate for any $i\in [n-1]$, namely $v_i'>
      v_i$;
  \item if $c_t<c_{t'}$ for some $t,t'\in M$, then we have $0\leqslant \gamma_{t'}^{\mathbf{v}}-\gamma_t^{\mathbf{v}}\leqslant c_{t'}-c_t$.
\end{enumerate}
We want to introduce functions $F_t^\mathbf{v}({x})$ (a special notation to be used only within Proposition \ref{pr:extext}) to become the desired values of $Q(\mathbf{v},x)$ on segments emerging from vertices of $V$, i.e., for all ${x} \in [a_t,b_t]$, all $t\in M$, and all $\mathbf{v}\in V$. We want these functions to satisfy conditions
{\it\begin{enumerate}[({\rm C}i)]
  \item Function $F_t^\mathbf{v}$ is increasing and 1-Lipschitz in its only argument on $[a_t,b_t]$ for every $t\in L_n$ and every $\mathbf{v}\in V$.
  \item These functions satisfy the boundary conditions $F_t^\mathbf{v}({a_t})= \alpha_t^\mathbf{v}= Q^*(\mathbf{v},a_t)$ and $F_t^\mathbf{v}({b_t}) =Q^*(\mathbf{v},b_t)$ for $t\in L_n$.
  \item If $a_t<a_{t'}$ for some $t,t'\in L_n$, then we have $0\leqslant F_{t'}^\mathbf{v}(x')-F_{t}^\mathbf{v}(x)\leqslant x'-x$ for all $\mathbf{v}\in V$ and $x\in[a_t,b_t],x'\in[a_{t'},b_{t'}]$.
  \item For every $t\in L_n$, every $x\in[a_t,b_t]$, and every $\mathbf{v},\mathbf{v}'\in V$ such that they are distinct only in the $i$-th coordinate, we have $0\leqslant F_t^{\mathbf{v}'}(x)-F_t^\mathbf{v}(x)\leqslant v_i'-v_i$.
\end{enumerate}}
We are now in position to extend our quasi-copula on these segments thus fulfilling the starting point of our inductive algorithm. Observe that the role of the set $M$ with indices taken from $L_n$ can be replaced by an analogous set with indices from $L_k, k\in[n]$, while replacing the mesh $V$ with the one having coordinates in the directions different from $k$. This way we can extend $Q^*$ on the segments, i.e., 1-dimensional faces of our grid.

\begin{proposition}\label{pr:extext}
  If the segment given by a $t\in M$ and a $\mathbf{v}\in V$ is a member of $\mathcal{I}_n$, then the desired function is uniquely determined by the sub-quasi-copula $Q^*$. If it is a member of $\mathcal{O}_n$, then
  the functions $\underline{F}_t^\mathbf{v}(x)=\max\{F_t^\mathbf{v}({a_t}), F_t^\mathbf{v}({b_t})+x-b_t\}$ and $\overline{F}_t^\mathbf{v}(x) =\min\{F_t^\mathbf{v}({a_t})+x-a_t,F_t^\mathbf{v}({b_t})\}$ satisfy Conditions \textit{({\rm C}i)--({\rm C}iv)} for $x\in[a_t,b_t]$ with $t\in M$. Furthermore, if ${F}_t^\mathbf{v}(x)$ are any functions satisfying these conditions, then
\[
    \underline{F}_t^\mathbf{v}(x)\leqslant{F}_t^\mathbf{v}(x)\leqslant \overline{F}_t^\mathbf{v}(x),\quad\mbox{for}\quad x\in[a_t,b_t]\quad \mbox{with}\quad t\in M.
\]
\end{proposition}


\begin{proof} Fix $\mathbf{v}\in V$ and assume that the segment corresponding to $t\in M$ is a member of $\mathcal{O}_n$. Write $F_t$ instead of $F_t^\mathbf{v}$ to simplify the notation and similarly with analogous notations with the lower and upper bars.
  (1) Let us show that $\underline{F}_t(x)$ satisfies the conditions. Conditions (C\emph{i}) and (C\emph{ii}) are easy. (C\emph{iii}): Assume $a_t<a_{t'}$, so that $a_t<b_t\leqslant a_{t'}<b_{t'}$. For $x,x'$ as in (C\emph{iii}) we have clearly (after adding and subtracting $\underline{F}_{t'}(a_{t'})$ and $\underline{F}_t(b_t)$)
  \[
     0\leqslant \underline{F}_{t'}(x')-\underline{F}_t(x)
     \leqslant x'-a_{t'}+a_{t'}-b_t+b_t-x=x'-x.
  \]
  (C\emph{iv}): Let us introduce $\Delta(x)= F_t^{\mathbf{v}'}(x)-F_t^\mathbf{v}(x)$. Denote by $x^{\mathbf{v}}$, respectively $x^{\mathbf{v'}}$, the unique value with $\alpha_t^{\mathbf{v}}= \beta_t^{\mathbf{v}}+x^{\mathbf{v}}-b_t$, respectively $\alpha_t^{\mathbf{v}'}= \beta_t^{\mathbf{v}'}+x^{\mathbf{v}'}-b_t$. If $x\leqslant\min\{x^{\mathbf{v}},x^{\mathbf{v'}}\}$, then $\Delta(x)= \alpha_t^{\mathbf{v}'}-\alpha_t^{\mathbf{v}}$ and we are done. If $x\geqslant\max\{x^{\mathbf{v}}, x^{\mathbf{v'}}\}$, then $\Delta(x)=\beta_t^{\mathbf{v}'}-\beta_t^{\mathbf{v}}$ and we are also done. If $x^{\mathbf{v'}}\leqslant x\leqslant x^{\mathbf{v}}$ then $\Delta(x)= \beta_t^{\mathbf{v}'}+x -b_t -\alpha_t^{\mathbf{v}}\in [\alpha_t^{\mathbf{v}'}-\alpha_t^{\mathbf{v}}, \beta_t^{\mathbf{v}'}-\beta_t^{\mathbf{v}}]\subseteq [0,\mathbf{v}_i'-\mathbf{v}_i]$. Finally, if $x^{\mathbf{v}}\leqslant x\leqslant x^{\mathbf{v'}}$, then a similar construction yields the same result.
 (2) The fact that $\overline{F}_t(x)$ satisfies the same conditions goes in the same way. (3) Choose any functions ${F}_t(x)$ satisfying Conditions \textit{({\rm C}i)--({\rm C}iv)}. Since $F_t(x)\leqslant F_t(b_t)=\beta_t$ and $F_t(x)-F_t(a_t)=F_t(x)-\alpha_t\leqslant x-a_t$, we conclude that $F_t(x)\leqslant \min\{\beta_t,\alpha_t+x-a_t\}=\overline{F}_t(x)$ and the proof of the other inequality goes in a similar way.
\end{proof}

Having now all the extensions to the segments corresponding to our mesh it is time for an inductive step. (From now on the notation $F$ will be used as in this section before Proposition \ref{pr:extext}.) Recall that the main conditions on a quasi-copula (after taking care of the boundary conditions) are that it be monotone and 1-Lipschitz in each argument. So, it is not hard to see that the following conditions are necessary for a sub-quasi-copula to be further extendable to a quasi-copula. We will write them in notation adjusted to $\widetilde{R}_\mathbf{s}^k$, where $k$ is the number of segments that belong to $\mathcal{I}$. (Observe that our induction is to be performed on $k$.)
\begin{enumerate}[(A)]
  \item In all variables $s_i$ for all $1\leqslant i\leqslant k$ the function $F$ must satisfy Conditions \textbf{PB} \emph{(i)--(iii)}.
  \item Let for any $j\in[n]$ and some $\mathbf{s},\mathbf{s}'\in L$ such that $s_i=s_i'$ for $i\in[n],i\ne j$, and $b_{s_j}\leqslant a_{s_j'}$. Then we have $0\leqslant F_{\mathbf{s}'}(\mathbf{x}^{z_j})- F_{\mathbf{s}}(\mathbf{x}^{z_j'}) \leqslant a_{s_j'}-b_{s_j}$ for all $\mathbf{x}\in D_F\subseteq R_\mathbf{s}$,\footnote{Observe that this is the same as requiring $\mathbf{x}\in R_{\mathbf{s}'}$. Consequently, $D_F$ are the same on both rectangles.} where $z_j,z_j'$ are defined as in Lemma \ref{le:additive}.
\end{enumerate}
Condition (B) implies, in particular, that if $b_{s_j}= a_{s_j'}$, then $F_{\mathbf{s}'}(\mathbf{x}^{z_j}) = F_{\mathbf{s}}(\mathbf{x}^{z_j'})$ for all $\mathbf{x}\in R_\mathbf{s}$. In other words, if two rectangles have all but one indeces in common, so that they intersect in an $(n-1)$-face, their functions $F$ have the same value on this intersection.

\begin{theorem}\label{th:main} Fix a sub-quasi-copula $Q^*$ and a rectangle $R_\mathbf{s},\mathbf{s}\in L$. 
\begin{enumerate}[(a)]
  \item Assume that for the rectangle considered the value of the hypothetical boundary function $F$ has been extended to all faces of dimension up to a certain index $k$ and choose a face of dimension $k+1$. Then the boundary function $F$ can be extended to this face satisfying Conditions (A) and (B).
  \item The function $F$ can be extended to a patch on the whole rectangle considered satisfying Condition (A) and (B). The collection of these patches (when going through all the rectangles in a certain order and taking into account the definitions of functions $F$ on previous rectangles) defines a quasi-copula extension of the given sub-quasi-copula, after extending it to the closure by continuity.
  \item Every quasi-copula extending the given sub-quasi-copula can be obtained in this way.
\end{enumerate}
\end{theorem}

\begin{proof}
  \emph{(a)} Suppose that for the rectangle considered the value of the hypothetical boundary function $F$ has been extended to all faces of dimension up to a certain index $k$ and choose a face of dimension $k+1$. Assume with no loss that the face belongs to the first $k+1$ coordinates; if not, apply an appropriate permutation on coordinates. Here is the description of the next inductive step needed. Using Proposition \ref{pr:mejekrp} we fill all the patches whose first $k+1$ coordinates correspond to this face and the rest of the coordinates belong to respective $L_i, k+1<i\leqslant n$. (The first $k+1$ coordinates play the role of parameter $t$ in Proposition \ref{pr:extext}, while the rest of the coordinates represent points of the mesh determining locations of the face, denoted there by $\mathbf{v}$.) Observe that in case $k+1=n$ this condition is empty; this step presents therefore the last inductive step which means that we are done. In any case we need to show that Conditions (A) and (B) extend to the next step (if $k<n+1$) or conclude the final fit into the boundary values provided by $Q^*$ and its previous extensions. Once we are done with this particular face, we proceed with other faces of this size. Since there are only countably many of them we can conclude this stage of the proof using an inner induction within this inductive step.

So, to conclude this part of the proof we only need to show that the two conditions are valid. Start by (B): Take $j\in[n]$ and some $\mathbf{s},\mathbf{s}' \in L$ such that $s_i=s_i'$ for $i\in[n],i\ne j$, and that $b_{s_j}\leqslant a_{s_j'}$
. Then we want to show that
\begin{equation}\label{eq:(B)}
  0\leqslant F_{\mathbf{s}'}(\mathbf{x}^{z_j})- F_{\mathbf{s}}(\mathbf{x}^{z_j'}) \leqslant a_{s_j'}-b_{s_j}
\end{equation}
for all $\mathbf{x}\in D_F\subseteq R_\mathbf{s}$. In order to prove that we need to recall the definition of the function $F$. We verify \eqref{eq:(B)} for $F$ coming from two sources: (1) The upper bound of all possible patches fitting into the $(k+1)$-dimensional face under consideration. We obtain it via the formula for $\overline{\mathcal{S}}$ of Proposition \ref{pr:mejekrp}, after specifying first the upper bounds for values on segments, then on 2-dimensional faces, and so on up to $k$-dimensional faces of the given face. (2) In a similar way, the lower bound of all possible patches fitting into the $(k+1)$-dimensional face under consideration is obtained via the formula for $\underline{\mathcal{S}}$ of Proposition \ref{pr:mejekrp}.
To show Condition \eqref{eq:(B)} in case (1) we observe that the two functions $F$ are given each by a minimum of terms made of functions $F$ from the previous step for which we know that an analogous condition holds. Any of the terms defining the right hand side function $F$ has a counterpart in the left hand side function $F$, so that the difference between these two terms is greater than or equal to $0$.
By replacing the right hand side term with the minimum of these terms we conclude that $0\leqslant \mathrm{term}- F_{\mathbf{s}}(\mathbf{x}^{z_j'})$, where notation ``term'' stands for any of the terms defining $F_{\mathbf{s}'}(\mathbf{x}^{z_j})$. So, when we take the minimum over all terms under consideration, we decide that the left hand side inequality of Condition \eqref{eq:(B)} is true. To get the right hand side inequality we reverse this procedure: to each defining term of the left hand side function $F$ there is a corresponding term in the right hand side function $F$ such that the difference between these two terms is no greater than $a_{s_j'}-b_{s_j}$. This time we first observe that this inequality prevails when we replace the first of the terms by $F_{\mathbf{s}'}(\mathbf{x}^{z_j})$ and than conclude the proof of Condition \eqref{eq:(B)} by taking the minimum over all the so obtained inequalities.
Case (2) of Condition (B) goes in a similar way
. Condition (A) requires that the first three of Conditions \textbf{PB} hold. While conditions \emph{(i)} and \emph{(ii)} are satisfied by constructon of functions $F$, Condition \emph{(iii)} amounts to a verification very similar to the one we just did under (B), so that we will omit it.

  In the proof of assertions \emph{(b)} and \emph{(c)}, which follow easily from \emph{(a)}, we need to be careful to combine first, say, the upper bounds of all patches. Since their union is dense in $\II^n$, we can extend this function by continuity to get the desired extension to a quasi-copula, an upper bound of all possible solutions. The lower bound is obtained in an analogous way and all the solutions are lying between the two.
%
\end{proof}


\begin{center}
  \textsc{    Acknowledgement    }
\end{center}

The authors are thankful for many fruitful discussions on some themes of this paper to the following colleagues: David Dol\v{z}an, Peter Erich Klement, Susanne Saminger-Platz, and Nik Stopar. 
{The authors acknowledge financial support from the ARIS (Slovenian Research and Innovation Agency, research core funding No. P1-0222).}


\bibliographystyle{amsplain}

\end{document}